\documentclass[10pt]{amsart}

\usepackage[T1]{fontenc}
\usepackage{lmodern}
\usepackage{microtype}
\usepackage{amsmath,amssymb,mathtools}
\usepackage[margin=1.05in]{geometry}
\usepackage[colorlinks=true,linkcolor=blue,citecolor=blue,urlcolor=blue]{hyperref}
\hypersetup{
  pdftitle={Automorphisms of a $24$-line configuration on the Schur quartic},
  pdfauthor={Gerald H\"ohn}
}

\newcommand{\Aut}{\operatorname{Aut}}
\newcommand{\Stab}{\operatorname{Stab}}
\newcommand{\NS}{\operatorname{NS}}
\newcommand{\PP}{\mathbf P}
\newcommand{\cC}{\mathcal C}
\newtheorem{theorem}{Theorem}
\theoremstyle{remark}
\newtheorem{remark}{Remark}

\title[Automorphisms of the Schur configuration]
{Automorphisms of a $(24_4, 32_3)$-configuration on the Schur quartic}
\author{Gerald H\"ohn}
\address{Department of Mathematics, Kansas State University, Manhattan, Kansas 66506, USA}
\date{July 10, 2026}

\begin{document}
\begin{abstract}
We answer the automorphism question raised by Naskr\k{e}cki and Pokora for their 
$(24_4, 32_3)$-configuration on the Schur quartic.  
The stabilizer of either $24$-line half is isomorphic to $W(D_4)\rtimes C_3$, where
$C_3$ acts by even triality; the full projective automorphism group has order~$1152$. The $D_4$ model also gives
an intrinsic coloring and shortens the incidence check.
\end{abstract}

\maketitle

Let
\[
 X=\{x_0^4-x_0x_1^3-x_2^4+x_2x_3^3=0\}\subset\PP^3,
\]
let $H$ be the hyperplane class, and let $D$ be the divisor formed by the
$24$ lines of \cite{NP26}.  Put
\[
 \alpha=\operatorname{diag}(1,1,i,i),\qquad D^*=\alpha(D),\qquad
 \cC_2=D+D^*.
\]
Thus $D\sim6H$, $\cC_2\sim12H$, and $\cC_2$ is the divisor of the $48$ lines
of the second kind \cite[Propositions~4.1 and 5.1]{NP26}.  Degtyarev identifies
the Schur quartic with the singular $K3$ surface $X([8,4,8])$
\cite[Table~1]{Deg19}; hence $\rho(X)=20$ and
\[
 T_X\cong
 \begin{pmatrix}8&4\\4&8\end{pmatrix}\cong A_2(4),
 \qquad
 \NS(X)\cong U\oplus E_8(-1)^{\oplus2}\oplus A_2(-4).
\]
The second isometry follows from unimodularity of the $K3$ lattice and
Nikulin's uniqueness theorem \cite[Theorem~1.14.2]{Nik80}.

\begin{theorem}
The abstract group $\Aut(X)$ is infinite, whereas
\[
 G:=\Aut(X,H)=\Aut_{\PP^3}(X)
   \cong T_{192}\rtimes C_6,
 \qquad |G|=1152.
\]
The pair $(X,G)$ is case $77a$ of \cite[\S6.7]{BH21}, with GAP identifier
$[1152,157515]$.  Moreover,
\[
 \Stab_{\Aut(X)}(\cC_2)=G
\]
and
\[
 G_D:=\Stab_{\Aut(X)}(D)=\Stab_{\Aut(X)}(D^*)
 \cong (T_{24}*T_{24})\rtimes C_2
 \cong W(D_4)\rtimes C_3,
 \qquad |G_D|=576.
\]
Here $T_{24}$ is the binary tetrahedral group, $*$ denotes central product,
and $C_3$ acts by even triality.  The action on the two halves gives
\[
 1\longrightarrow G_D\longrightarrow G\longrightarrow C_2\longrightarrow1.
\]
\end{theorem}

\begin{proof}
The infinitude of $\Aut(X)$ and the abstract structure of $G$ follow from
\cite[\S1 and \S6.7]{BH21}; that classification uses, in particular, the
Leech-lattice data of H\"ohn--Mason \cite{HM16}.  Nukui's explicit isomorphism
with the Schur model \cite[Proposition~3.6 and \S4]{Nuk26} gives the following
generators
of $G$.  Write $\tau^2=-3$,
$\omega=(-1+\tau)/2$, and split the coordinates as
$(x_0,x_1)\mid(x_2,x_3)$.  Then
\[
 g_\gamma=\begin{pmatrix}\gamma&0\\0&I_2\end{pmatrix},\quad
 g_\delta=\begin{pmatrix}\delta&0\\0&I_2\end{pmatrix},\quad
 s=\begin{pmatrix}0&I_2\\I_2&0\end{pmatrix},\quad
 q=\begin{pmatrix}iI_2&0\\0&I_2\end{pmatrix},
\]
where
\[
 \gamma=\begin{pmatrix}1&0\\0&\omega\end{pmatrix},\qquad
 \delta=\frac{\tau}{3}\begin{pmatrix}1&-1\\-2&-1\end{pmatrix}.
\]
Substitution in the line parametrizations of \cite[Appendix~A]{NP26} gives
\[
 g_\gamma(D)=g_\delta(D)=s(D)=D,\qquad q(D)=D^*.
\]
The first two matrices generate a copy of $T_{24}$; conjugation by $s$ gives
a second copy, their centers are identified in $\operatorname{PGL}_4$, and
$s$ interchanges the factors.  Hence
$\langle g_\gamma,\, g_\delta,\, s\rangle\cong(T_{24}*T_{24})\rtimes C_2$ has
order $576$ and is precisely $G_D$.

If $u\in\Aut(X)$ stabilizes $D$, then
$6\,u^*H=u^*D=D=6H$; since $\NS(X)$ is torsion free, $u^*H=H$.
The same argument with $\cC_2\sim12H$ excludes additional nonprojective
stabilizers of $\cC_2$.  Conversely, the displayed generators preserve
$\cC_2$, proving the stabilizer assertions.

For the second description of $G_D$, label the roots
$\Phi(D_4)=\{\pm e_a\pm e_b:1\leq a<b\leq4\}$ by the lines as follows:
\[
\renewcommand{\arraystretch}{1.10}
\begin{array}{c|cccc}
12&L_2:e_1+e_2&L_{15}:e_1-e_2&L_{16}:-e_1+e_2&L_1:-e_1-e_2\\
13&L_6:e_1+e_3&L_{11}:e_1-e_3&L_{20}:-e_1+e_3&L_3:-e_1-e_3\\
14&L_7:e_1+e_4&L_8:e_1-e_4&L_{23}:-e_1+e_4&L_{24}:-e_1-e_4\\
23&L_9:e_2+e_3&L_5:e_2-e_3&L_4:-e_2+e_3&L_{22}:-e_2-e_3\\
24&L_{12}:e_2+e_4&L_{10}:e_2-e_4&L_{21}:-e_2+e_4&L_{19}:-e_2-e_4\\
34&L_{13}:e_3+e_4&L_{14}:e_3-e_4&L_{17}:-e_3+e_4&L_{18}:-e_3-e_4.
\end{array}
\]
The $32$ triples of \cite[Table~3]{NP26} are exactly the unordered triples
$\{r_1,\,r_2,\,r_3\}$ with $r_1+r_2+r_3=0$.  For distinct roots $r$, $r'$, the
number of common neighbors in this triple geometry is $0$, $1$, $4$, $3$ according as
$(r,r')=-2$, $-1$, $0$, $1$.  Thus the incidence structure recovers the root inner
product and
\[
 \Aut_{\mathrm{inc}}(D)=\Aut\Phi(D_4)
   =W(D_4)\rtimes S_3\cong W(F_4).
\]
The permutations induced by $g_\gamma$, $g_\delta$ and~$s$ form its index-two
subgroup $W(D_4)\rtimes C_3$, the even-triality subgroup.
\end{proof}

\begin{remark}[Intrinsic coloring and two shortcuts]
In the triple-incidence graph of the $48$ second-kind lines, no triple point
mixes the two halves and each half is connected
\cite[Propositions~3.5 and~4.1]{NP26}.  Hence the two connected components are
$D$, $D^*$; this gives the intrinsic coloring requested in
\cite[Section~8]{NP26}, uniquely up to interchange.  Moreover, $G_D$ has four
orbits on unordered pairs of lines, of sizes $12$, $72$, $96$, $96$, corresponding to
root inner products $-2$, $0$, $1$, $-1$, and is transitive on the $32$ zero-sum
triples.  Thus, after checking the three generator actions, four determinants
replace the $\binom{24}{2}=276$ pair tests in
\cite[Proposition~3.5]{NP26}, and one concurrency check handles all $32$
triples.
\end{remark}

\bibliographystyle{amsalpha}

\providecommand{\bysame}{\leavevmode\hbox to3em{\hrulefill}\thinspace}
\providecommand{\MR}{\relax\ifhmode\unskip\space\fi MR }
\providecommand{\MRhref}[2]{%
  \href{http://www.ams.org/mathscinet-getitem?mr=#1}{#2}
}
\providecommand{\href}[2]{#2}

\end{document}